\RequirePackage{fix-cm}
\documentclass[smallextended]{svjour3}       
\smartqed  
\usepackage{graphicx}
\usepackage{tristan}

\usepackage{wasysym}
\usepackage{ragged2e}

\renewcommand{\enorm}[1]{\ensuremath{\Norm{#1}}_{dG}}

\usepackage{a4wide}

\newcommand{\del}{\partial}

\renewcommand{\vec}[1]{\geovec{#1}}

\setboolean{showtodo}{false}
\setboolean{shownotes}{false}
\setboolean{showchanges}{false}
\setboolean{usemathrsfs}{false}
\numberwithin{equation}{section}

\begin{document}

\title{Reduced relative entropy techniques for a priori analysis of
  multiphase problems in elastodynamics
\thanks{J.G. was partially supported by the German Research Foundation
    (DFG) via SFB TRR 75 `Tropfendynamische Prozesse unter extremen
    Umgebungsbedingungen'. T.P. was partially supported by the EPSRC grant EP/H024018/1 and an LMS travel grant 41214.}
}


\author{Jan Giesselmann
  \and
  Tristan Pryer
}


\institute{Jan Giesselmann  \at
  Institute of Applied Analysis and
    Numerical Simulation, University of Stuttgart,
    Pfaffenwaldring 57, D-70563 Stuttgart, Germany.
    {\tt{jan.giesselmann@mathematik.uni-stuttgart.de}}
      \and
      Tristan Pryer \at
      Department of Mathematics and
      Statistics, Whiteknights, PO Box 220,
      Reading, GB-RG6 6AX, England UK. 
      {\tt{T.Pryer@reading.ac.uk}}
}

\date{\today}

\maketitle

\begin{abstract}
  We give an a priori analysis of a semi-discrete discontinuous
  Galerkin scheme approximating solutions to a model of multiphase
  elastodynamics which involves an energy density depending not only
  on the strain but also the strain gradient. A key component in the
  analysis is the \emph{reduced relative entropy} stability framework
  developed in [Giesselmann 2014]. We prove optimal bounds for the strain and the velocity in
  an appropriate norm.
  \keywords{discontinuous Galerkin finite element method, a priori error
  analysis, multiphase elastodynamics, relative entropy, reduced
  relative entropy.}
  \subclass{65M60, 65M12, 65M15, 74B20.}
\end{abstract}

\section{Introduction}

Our goal in this work is to introduce the \emph{reduced relative
  entropy} technique as a methodology for deriving a priori error
estimates to finite element approximations of a problem arising in
elastodynamics. In particular, this work is concerned with providing a
rigorous a priori error estimate for a semi (spatially) discrete
discontinuous Galerkin scheme approximating the solution of a
multiphase problem in nonlinear elasticity.  We consider a model for
shearing motions in an elastic bar undergoing phase transitions
between phases corresponding to different (intervals of shear)
strains.  The model is based on the equations of nonlinear
elastodynamics with a non-convex energy density regularized by an
additional (quadratic) dependence of the energy density on the strain
gradient.  Such models are frequently called ''second (deformation)
gradient`` models \cite{JTB02}.  It should be noted that (due to the
non-convexity of the energy) it is not immediately obvious what an
appropriate stability theory is.  A possible answer to this question
was given in \cite{Gie} where a modification of the relative entropy
approach was presented, which uses the higher order regularizing terms
in order to compensate for the non-convexity of the energy.

The relative entropy framework for hyperbolic conservation laws
endowed with a convex entropy was introduced in \cite{Daf79,DiP79}.
For systems of conservation laws describing (thermo)-mechanical
processes the notion of (mathematical) entropy follows from the
physical one \cite{Daf10}.  The generalization of the relative entropy
techniques to entropies which are quasi or polyconvex is by now
standard and is discussed in detail in \cite{Daf10}.  It should be
noted, however, that the model considered in this study does not fall
into this framework which requires us to build our analysis around the
stability framework from \cite{Gie}.

Our analysis is based on deriving a space discrete version of the
modified relative entropy framework from \cite{Gie}.  This enables us
to derive an estimate for the difference of solutions to our
semi-discrete scheme and a perturbed version thereof.  We combine this
stability framework with appropriate projection operators which enable
us to show that the exact solution satisfies a perturbed version of
the numerical scheme.

In order to be more precise let us introduce the equations under consideration:
In one space dimension the equations of nonlinear elasticity read
\begin{equation}\label{firstorder}
 \begin{split}
 \del_t u - \del_x v &=0\\
  \del_t v - \del_x W'(u)&=0,
 \end{split}
\end{equation}
where $u$ is the strain, $v$ is the velocity and $W=W(u)$ is the
energy density given by a constitutive relation.  They can also be
cast as a nonlinear wave equation for the deformation field $y$
satisfying $\del_x y=u:$
\[
 \del_{tt} y - \del_x(W'(\del_x y)) =0.
\]
A priori estimates for continuous finite element and dG schemes approximating the wave equation 
can be found in \cite{Mak93,KM05,OS07}.
For \eqref{firstorder} to describe multiphase behaviour the energy
density $W$ needs to be non-convex which makes \eqref{firstorder} a
problem of mixed hyperbolic-elliptic type.  This leads to many
problems concerning e.g. uniqueness of solutions to
\eqref{firstorder}.  To overcome the difficulties caused by the
hyperbolic-elliptic structure either a kinetic relation \cite{AK91,LeF02} needs to be introduced, or regularizations of
\eqref{firstorder} need to be considered.  We will study the numerical
approximation of systems arising from the second approach. In
particular, we will study the following regularized problem which was
considered by many authors \cite[e.g.]{ERV13,LT02,AB82,HL00,Sle84,Sle83}:
\begin{equation}\label{thorder}
 \begin{split}
  \del_t u - \del_x v &=0\\
  \del_t v - \del_x W'(u)&=\mu \del_{xx}v - \gamma \del_{xxx} u,
 \end{split}
\end{equation}
where $\mu\geq0,\, \gamma>0$ are parameters which scale the strength
of viscous and capillary effects.  It should be noted that
\eqref{thorder} is a physically meaningful model in itself, which also
can be written in wave equation form
\begin{equation}\label{eq:wave}
 \del_{tt} y -\del_x(W'(\del_x y)) = \mu \del_{xxt} y - \gamma \del_{xxxx} y.
\end{equation}

The numerical simulation of the model at hand and similar models, like
the Navier-Stokes-Korteweg system, has received some attention in
recent years \cite[e.g.]{CL01,BP13,Die07,JTB02,GiesselmannMakridakisPryer:2014,TXKV14}. Indeed it
turned out that stability of numerical solutions is not easy to
obtain. In \cite{TXKV14} an a priori analysis is carried out under the assuption that $W$ is linear.
 We are interested in the case that $\gamma$ is small. In this
case it is expected that solutions of \eqref{thorder} display thin
layers at phase boundaries. Thus, we advocate the use of
discontinuous Galerkin (dG) finite element methods. 

The remainder of the paper is organized as follows: After giving some
basic definitions we study well-posedness of \eqref{thorder} and its
associated energy in \S\ref{sec:re}. In \S\ref{sec:ns} we define the
semi-discrete dG scheme and describe some immediate properties of the
involved (discrete) operators.  In \S\ref{sec:dre} we derive a
discrete version of the reduced relative entropy framework and derive
a stability estimate for solutions of a perturbed version of the
numerical scheme.  \S\ref{sec:proj} is devoted to the
construction of projection operators.  The aim is to show that the
projection of the exact solution of \eqref{thorder} is a solution to a
perturbed version of our dG scheme.  In order to derive the projection
operators we need to study the gradient operators used in the dG
scheme in more detail. We combine the results of the
preceding sections in \S\ref{sec:ee} in order to derive an
error estimate for our dG scheme. Finally in \S\ref{sec:num} we conduct some numerical benchmarking experiments.

\section{Preliminaries, well-posedness and relative entropy}
\label{sec:re}

Given the standard Lebesgue space notation
\cite{Ciarlet:1978,Evans:1998} we begin by introducing the Sobolev
spaces. Let $\W \subset \reals$ then
\begin{gather}
  \sob{k}{p}(\W)
  := 
  \ensemble{\phi\in\leb{p}(\W)}
  {\D^{\alpha}\phi\in\leb{p}(\W), \text{ for } \norm{\alpha}\leq k},
\end{gather}
which are equipped with norms and seminorms
\begin{gather}
  \Norm{u}_{\sob{k}{p}(\W)}
  := 
  \begin{cases}
    \qp{\sum_{\norm{\alpha}\leq k}\Norm{\D^{\alpha} u}_{\leb{p}(\W)}^p}^{1/p} &\text{ if } p \in [1,\infty)
    \\
    \sum_{\norm{\alpha}\leq k}\Norm{\D^{\alpha} u}_{\leb{\infty}(\W)} &\text{ if } p = \infty 
  \end{cases}
  \\
  \norm{u}_{\sob{k}{p}(\W)}
  :=
  \Norm{\D^k u}_{\leb{p}(\W)} 
\end{gather}
respectively, where derivatives $\D^{\alpha}$ are understood in a
weak sense. 

We also make use of the following notation for time dependent Sobolev
(Bochner) spaces:
\begin{gather}
  \cont{i}(0,T; \sobh{k}(S^1))
  :=
  \ensemble{u : [0,T] \to \sobh{k}(S^1)}
           {u \text{ and $i$ temporal derivatives are continuous}},
           \\
           \leb{\infty}(0,T; \sob{k}{p}(\W))
  :=
           \ensemble{u : [0,T] \to \sob{k}{p}(\W)}
                    {\esssup_{t\in [0,T]} \Norm{u(t)}_{\sob{k}{p}(\W)} < \infty}.
\end{gather}
We define $\sobh{k}(\W):= \sob{k}{2}(\W).$ For any function space
 the subspace of functions with vanishing mean is denoted by subscript $m.$

We complement \eqref{thorder} with periodic boundary conditions. To
make this obvious in the notation we consider \eqref{thorder} on
$[0,T) \times S^1$ for some $T>0$ where $S^1$ denotes the flat circle,
\ie the interval $[0,1]$ with the endpoints being identified with each
other.  We also need an initial condition $u(0,\cdot)=u_0$ for some
$u_0 : S^1 \rightarrow \mathbb{R}$ whose regularity we will specify
later.

We assume $W \in \cont{3}(\rR, [0,\infty))$ but make no assumption on
the convexity of $W$. The standard application we have in mind is that
$W$ has a multi-well shape.

The well-posedness of \eqref{thorder} can be ensured using semi-group
theory:
\begin{proposition}[Well-posedness]\label{prop:thorder}
  Let $k \in \rN,\, k \geq 3$ and initial data $u_0 \in \sobh{k}(S^1),
  \, v_0 \in \sobh{k-1}(S^1)$ with $\int_{S^1}u_0 \d x= \int_{S^1} v_0
  \d x=0$ and $\mu,\gamma>0$ be given. Let $W \in \cont{k}(\rR).$
  Then, there exists some $T>0$ such that the problem \eqref{thorder}
  has a unique strong solution $(u,v)$ satisfying
\[\begin{split} u & \in \cont{0}([0,T],\sobh{k}(S^1))\cap  \cont{1}((0,T),\sobh{k-2}(S^1)) \\
   v & \in \cont{0}([0,T],\sobh{k-1}(S^1))\cap  \cont{1}((0,T),\sobh{k-3}(S^1))
  \end{split}
 \]
with $\int_{S^1}u(t,\cdot) \d x= \int_{S^1} v(t,\cdot) \d x=0$ for all $0 \leq t \leq T.$
\end{proposition}

In case $k=3$ the solution exists for arbitrary times $T>0$. This,
indeed, relies on the compatibility of the model with the second law
of thermodynamics, \ie the following energy dissipation equality which
is well-known.
\begin{lemma}[Energy balance for \eqref{thorder}]
  Let $T,\,\gamma >0$ and $\mu\geq 0$ be given and let
  \begin{equation}
    \begin{split}
      \qp{u,v}
      \in 
      &\qp{
        \cont{0}([0,T],\sobh{3}(S^1))\cap
        \cont{1}((0,T),\sobh{1}(S^1))
      }
      \times
      \qp{
        \cont{0}([0,T],\sobh{2}(S^1))\cap
        \cont{1}((0,T),\leb{2}(S^1))
      }
    \end{split}
  \end{equation}
  be a strong solution of \eqref{thorder}. Then, the following energy
  balance law holds in $(0,T) \times S^1:$
  \begin{equation}\label{to_energy}
    \begin{split}
      0 &= \pd {t} 
      {
        \qp{
          W(u) + \frac{\gamma}{2} 
          \qp{ \pd x u}^2 + \frac{1}{2} v^2
        }
      }
      -
      \pd x {
        \qp{ v W'(u) - \gamma v  \pd {xx} u
          + \gamma \pd x v \pd x u + \mu v \pd x v 
        } 
      }
      + \mu (\pd x v)^2.
    \end{split}
  \end{equation}
\end{lemma}

\begin{proof}{of Proposition \ref{prop:thorder}.}
  The result for $k=3$ can be found in \cite{Gie}. We will show the
  result for $k=4$, the generalization to $k \geq 5$ is
  straightforward.  Note that by forming the $x$-derivative of
  \eqref{eq:wave} we obtain the following equation for $u=\pd x y$
  \begin{equation}\label{newwave} 
    \del_{tt} u - \del_x (W''(\del_x y) \del_x u)
    =
    \mu \del_{xxt} u - \gamma \del_{xxxx} u
  \end{equation}
  where $\del_x y$ is considered to be already given (from the result
  for $k=3$).  With $z= \Transpose{ (u, \del_t u)}$ this can be cast in
  abstract form as
  \begin{equation}
    \label{abst}
    \del_t z = A z +f(z) 
    \
    \text{ with }
    A = \begin{pmatrix}
      0 & \operatorname{Id} 
      \\
      -\gamma \del_{xxxx} & \mu \del_{xx}
    \end{pmatrix}
    \quad f(z) = \begin{pmatrix}
      0 \\ \del_x(W''(\del_x y) \del_x z_1)
    \end{pmatrix}.
  \end{equation}
  Let us define the spaces 
  \begin{equation}
    X:=  \sobh{2}_m(S^1) , \quad Y:= X \times L^2(S^1).
  \end{equation}
  For every $w \in X$ it holds that $\del_x w \in \sobh{1}_m(S^1)$
  such that, by Poincar\'{e}'s inequality,
  \begin{equation}
    \left\langle \begin{pmatrix}
        z_1\\ z_2
      \end{pmatrix}, 
      \begin{pmatrix}
        \tilde z_1\\ \tilde z_2
      \end{pmatrix}\right\rangle_Y := \int_{S^1} \gamma \del_{xx} (z_1) \del_{xx} (\tilde z_1) + z_2 \tilde z_2 \d x, \quad 
    \left\| \begin{pmatrix} z_1\\ z_2\end{pmatrix}\right\|_Y^2:=  \left\langle \begin{pmatrix}
        z_1\\ z_2
      \end{pmatrix}, 
      \begin{pmatrix}
        z_1\\ z_2
      \end{pmatrix}\right\rangle_Y
  \end{equation}
  define a scalar product and a norm on $Y.$ The operator $A$ is
  densely defined on $Y$ with
  \begin{equation}
    D(A) = \big(\sobh{4}(S^1) \cap X\big) \times \sobh{2}(S^1).
  \end{equation}
  The operator $A$ induces a $C^0$ semi-group on $Y$ which can be seen
  analogously to the arguments in \cite{AB82} using $\{ \sin{2n \pi
    \cdot}, \cos{2n \pi \cdot} \, : \, n \in \rN\}$ as a basis of $X$.
  Note that for all $t \geq 0$ it holds that
  \[ \int_{S^1} u(t,\cdot) \d x =0, \quad \int_{S^1} \del_x u(t,\cdot)
  \d x =0, \quad \int_{S^1} \del_t u (t,\cdot) \d x =0, \] due to our
  assumptions on the initial data and the fact that the wave equation
  \eqref{newwave} can be recast as conservation laws for $\del_x u,
  \del_t u$.  The semi-group induced by $A$ is, in fact, contractive
  as any solution $(z_1,z_2)$ of
  \begin{equation}
    \pd t {\begin{pmatrix} z_1\\ z_2 \end{pmatrix}} = A \begin{pmatrix} z_1\\ z_2 \end{pmatrix}
  \end{equation}
  satisfies
  \begin{equation}
    \begin{split}
      \frac{d}{dt} \left\| 
        \begin{pmatrix}
          z_1\\ z_2 
        \end{pmatrix}
      \right\|^2_Y 
      &=
      2\int_{S^1} \gamma \del_{xx} z_1 \del_{xxt} z_1 + z_2 \del_t z_2 \d x
      \\
      &=
      2\int_{S^1} \gamma \del_{xx} z_1 \del_{xxt} z_1- \gamma
      \del_{xxxx} z_1 \del_t z_1 + \mu z_2 \del_{xx} z_2 \d x 
      \\
      &= -
      2\int_{S^1} \mu (\del_x z_2)^2 \d x \leq 0 .
    \end{split}
  \end{equation}
  Moreover, the map $f : Y \rightarrow Y$ is locally Lipschitz
  continuous, as estimates for $\Norm{y}_{\sobh{2}(S^1)}$ are already
  known from the result for $k=3$.  Invoking \cite[Thm. 5.8]{Pav87} we
  infer that it exists a maximal time of existence $T_{m} \in
  (0,\infty]$ and a unique strong solution $(z_1,z_2)$ of
  \eqref{newwave} with
  \begin{equation}
    \begin{split}
      z_1 &\in C^0([0,T_{m}),\sobh{4}_m(S^1)) \cap C^1 ((0,T_{m}),\sobh{2}_m(S^1)),\\
      z_2 &\in C^0([0,T_{m}),\sobh{2}_m(S^1)) \cap C^1((0,T_{m}),\leb{2}(S^1)).
    \end{split}
  \end{equation}
  Now that we have obtained $z_1$ we may define some $\tilde y$ as the
  primitive of $z_1$ with mean value zero.  It is straightforward to
  check, by integrating \eqref{newwave}, that $\tilde y$ indeed solves
  \eqref{eq:wave}. As the solution of \eqref{eq:wave} is unique we have
  $y = \tilde y$ which implies $z_1 = \del_x y.$ This induces the
  desired additional regularity of $y.$
  
  The equations for higher spatial derivatives of $y$ can be obtained
  analogously to \eqref{newwave} and the arguments can be modified
  in a straightforward fashion.
\end{proof}

\section{Semi-discrete dG scheme}
\label{sec:ns}
We consider the approximation of \eqref{thorder} by a semi-discrete
discontinuous Galerkin scheme. To define the scheme let us first
introduce some standard notation: Let $I:=[0,1]$ be the unit interval
and choose $0 = x_0 < x_1 < \dots < x_N = 1.$ We denote
$I_n=[x_n,x_{n+1}]$ to be the $n$--th subinterval and let $h_n:=
x_{n+1}-x_n$ be its size. By $\mathfrak{h}$ we denote the mesh-size function $S^1 \rightarrow [0,\infty).$, \ie $\mathfrak{h}|_{I_n}=h_n$ and $h:= \max h_n.$ For the purposes of this work, we will assume that $h N \leq C$ for some $C>0$.  For $q \geq 1$ let $\poly q(I)$ be the space
of polynomials of degree less than or equal to $q$ {on $I$}, then we
denote
\begin{equation}
  \fes_q
  :=
  \ensemble{g : I \to \rR }
  { g \vert _{I_n} \in \poly q{(I_n)} \text{ for } \ n =0,\dots, N-1},
\end{equation}
to be the usual space
of piecewise $q$--th order polynomials for functions
over $I$.
By 
\begin{equation}
  \fes_q^m
  :=
  \ensemble{g \in \fes_q}{ \int_{S^1} g \d x=0},
\end{equation}
we denote the subspace of functions with vanishing mean.
In addition we define jump and average operators by
\begin{equation}
  \begin{split}
    \jump{g}_n
    &:= 
    g(x_n^-) - g(x_n^+)
    := 
    \lim_{s \searrow 0} g(x_n-s) - \lim_{s \searrow 0}  g(x_n+s),
    \\
    \avg{g}_n
    &:
    = 
    \frac{1}{2} \qp{ g(x_n^-) +g(x_n^+)}
    :=
    \frac{1}{2} \qp{\lim_{s \searrow 0} g(x_n-s) + \lim_{s \searrow 0} g(x_n+s)}.
  \end{split}
\end{equation}
We will also denote the $\leb{2}$ projection operator from $L^2(S^1)$
to $\fes_q$ by $P_q$.

We will examine semi-discrete numerical schemes which are based on the
following reformulation of \eqref{thorder} using an auxiliary variable
$\tau$:
\begin{equation}
  \begin{split}
    \del_t   u - \del_x v&=0\\
    \del_t v - \del_x \tau - \mu \del_{xx}  v&=0 \\
    \tau -  W'(u) + \gamma \del_{xx}  u &=0.            
  \end{split}
\end{equation}
In the semi-discrete numerical scheme the quantities $ u_h, v_h \in
\cont{1}([0,T),\rV_q)$ and $ \tau_h \in \cont{0}([0,T),\rV_q)$ are
determined such that
\begin{equation}\label{sds}
  \begin{split}
    \int_{S^1} \del_t u_h \Phi - G^-[v_h]\Phi \d x 
    &=0 
    \quad \forall \ \Phi \in \rV_q, \\
    \int_{S^1}\del_t v_h \Psi- G^+[\tau_h]\Psi  + \mu G^-[v_h]G^-[\Psi]\d x
    &=0
    \quad \forall \ \Psi \in \rV_q, \\
    \int_{S^1}\tau_h Z - W'(u_h) Z \d x - \gamma a_h^d(u_h,Z)
    &=0
    \quad \forall \ Z \in \rV_q,
  \end{split}
\end{equation}
given the initial conditions $u_h(0,\cdot) =P_q [u_0],
v_h(0,\cdot)=P_q[v_0],$ where $P_q$ is the $\leb{2}$ projection
$\leb{2}(S^1) \rightarrow \fes_q.$ In \eqref{sds} $G^\pm: \rV_q
\rightarrow \rV_q$ denote discrete gradient operators and $a_h^d:
\rV_q \times \rV_q \rightarrow \rR$ is a symmetric, bilinear form
which is a consistent discretisation of the weak form of the
Laplacian.  We will describe our assumptions on $a_h^d$ below.  For
any $w \in \fes_q$ the discrete gradients $G^\pm[w]$ are defined by
\begin{equation}
  \label{dgrad}
  \int_{S^1} G^\pm[w]\Psi \d x 
  =
  \sum_{i=0}^{N-1} \int_{x_i}^{x_{i+1}} \del_x w \Psi \d x 
  -
  \sum_{i=0}^{N-1} \jump{w}_{i} \Psi(x_i^\pm)
  \quad 
  \forall \ \Psi \in \rV_q,
\end{equation}
where the periodic boundary conditions are accounted for by
$\jump{w}_0 := w(x_N^-)-w(w_0^+).$

In the sequel we will use the convention that $C > 0$ denotes a generic constant which may depend on $q$, the ratio of concurrent cell sizes, $\gamma$, $W$, but is independent of $\mathfrak h$ and the exact solution $\qp{u,v}$. We impose that the bilinear form $a_h^d$ is coercive and stable with respect to the
dG-norm, \ie there exists a $C>0$ such that for all $w,\tilde w \in
\rV_q$
\begin{equation}\label{a:ahd}
 \begin{split}
   a_h^d(w, \tilde w) &\leq C \| w\|_{\operatorname{dG}} \|\tilde  w\|_{\operatorname{dG}}, \\
   | w|_{\operatorname{dG}}^2 &\leq C  a_h^d(w, w) ,
 \end{split}
\end{equation}
where
\begin{equation}
 \begin{split}
  \norm{ w }_{\operatorname{dG}}^2 &:= \sum_{n=0}^{N-1} \qp{\| \del_x w\|_{\leb{2}(I_n)}^2 + \frac{2\qp{\jump{w}_n}^2}{h_{n-1} + h_n} },\\ 
   \Norm{ w }_{\operatorname{dG}}^2 &:=  \Norm{w}_{\leb{2}(S^1)}^2 + \norm{ w }_{\operatorname{dG}}^2.
 \end{split}
\end{equation}

A classical choice for $a_h^d$ satisfying \eqref{a:ahd} is the
interior penalty method
\begin{equation}
  \label{def:ip}
  a_h^d(w , \tilde w )
  := 
  \sum_{i=0}^{N-1} \Big( \int_{x_i}^{x_{i+1}} \del_x w \del_x \tilde w \d x 
  -
  \jump{w}_{i} \avg{\del_x \tilde w}_{i}
  -
  \jump{\tilde w}_{i} \avg{\del_x w}_{i} + \frac{\sigma}{h} \jump{w}_{i} \jump{\tilde w}_{i}\Big),
\end{equation}
\newcommand{\ritz}{\mathfrak P}
for some $\sigma \gg 1,$ and $\avg{\del_xw}_0:= \tfrac{1}{2} (
\del_xw(x_N^-)+\del_xw(x_0^+)).$ In addition, we need $a_h^d$ to
satisfy the following approximation property. For some $w \in
\sobh{2}(S^1)$ let $\ritz [w]$ be the Riesz projection of $w$ with respect
to $a_h^d$, \ie the unique function in $\rV_q$ satisfying
\begin{equation}\label{Riesz:def}
  a_h^d(\ritz [w],\Psi)
  =
  \int_{S^1} \del_{xx} w \Psi \d x 
  \quad 
  \forall \ \Psi \in \rV_q 
  \quad 
  \text{and}
  \quad
  \int_{S^1} \ritz [w] - w \d x =0.
 \end{equation}
 We impose on $a_h^d$ that for every $w \in \sobh{q+2}(S^1)$ we have
 \begin{equation}\label{Riesz:est}
   \begin{split}
     | w - \ritz [w] |_{\operatorname{dG}} &\leq C h^{q} \Norm{w}_{\sobh{q+1}(S^1)}
     \\
     \| w - \ritz [w] \|_{\leb{2}(S^1)} &\leq C h^{q+1}\Norm{w}_{\sobh{q+1}(S^1)}
     \\
     \Norm{\ritz [w]}_{\sob{1}{\infty}(S^1)} &\leq C \Norm{w}_{\sob{1}{\infty}(S^1)}.
\end{split}\end{equation}
These conditions are also satisfied by the interior penalty method
\eqref{def:ip}, see \cite[Cor. 4.18, Thm. 4.25]{dPE12} and
\cite[Thms. 5.1, 5.3]{CC05}.

Let us note some properties of the discrete gradient operators, which
follow from \cite[Prop. 4.4]{GiesselmannPryer:2014b} and by standard inverse and trace
inequalities
\begin{lemma}[Properties of discrete gradients]\label{prop:grad1}
  The discrete gradients $G^\pm$ have the following duality property:
  \begin{equation}\label{dual}
    \int_{S^1} G^+ [\Phi] \Psi \d x
    =
    - \int_{S^1} \Phi G^-[\Psi] \d x \Foreach \Phi, \Psi \in \rV_q.
  \end{equation}
  The discrete gradients $G^\pm$ have the following stability
  property: For all $q \in \rN$ there exists $C>0$ independent of $h$
  such that
  \begin{equation}
    \label{eq:inverse-Gs}
    \begin{split}
      \Norm{G^\pm [\Phi]}_{\leb{2}(S^1)} &\leq C\Norm{\mathfrak{h}^{-1}\Phi}_{\leb{2}(S^1)}\\
      \Norm{G^\pm [\Phi]}_{\leb{2}(S^1)} &\leq C\norm{\Phi}_{\operatorname{dG}}
    \end{split}
    \Foreach \Phi \in \rV_q.
  \end{equation}
\end{lemma}
\begin{proof}
  The proof of (\ref{dual}) follows immediately from the definition of
  $G^\pm[\cdot]$, indeed
  \begin{multline}
      \int_{S^1} G^+[\Psi] \Phi 
      =   
      \sum_{i=0}^{N-1} \int_{x_i}^{x_{i+1}} \del_x \Psi \Phi \d x 
      -
      \sum_{i=0}^{N-1} \jump{\Psi}_{i} \Phi(x_i^+)
      \\
      =
      - \sum_{i=0}^{N-1} \int_{x_i}^{x_{i+1}} \Psi \del_x \Phi \d x 
      +
      \sum_{i=0}^{N-1} {\Psi(x_i^-)} \jump{\Phi}_i
      = 
      -\int_{S^1} \Psi G^-[\Phi].
    \end{multline}
  The proof of (\ref{eq:inverse-Gs}) uses standard inverse inequalities.
\end{proof}

\begin{remark}[Discrete entropy inequality]\label{rem:ed}
  Using the test functions $\Phi=\tau_h$, $\Psi=v_h$ and $Z=\del_t
  u_h$ in \eqref{sds} and employing the duality \eqref{dual} it is
  straightforward to see that our semi-discrete scheme satisfies the
  following entropy dissipation equality for $0 < t <T$
  \[ \frac{\d }{\d t} \Big(\int_{S^1} W(u_h) + \frac{1}{2} v_h^2\d x +
  \frac{\gamma}{2} a_h^d (u_h,u_h) \Big)= -\mu \|
  G^-[v_h]\|^2_{\leb{2}(S^1)} .\] The reader may note that this is
  similar to the entropy dissipation equality obtained in the fully
  discrete case in \cite{GiesselmannMakridakisPryer:2014}.  However there are also differences:
  In \cite{GiesselmannMakridakisPryer:2014} the authors required the dissipative term to be
  coercive (with respect to the dG-norm) and ``central'' discrete
  gradients were used instead of the one sided versions $G^\pm$ here.
\end{remark}

\begin{remark}[$\leb{\infty}$ bound for $u_h.$]
  As the numerical scheme dissipates discrete energy, $a_h^d$ is
  coercive, see \eqref{a:ahd},  $(\rV_q ,
  \|\cdot\|_{\operatorname{dG}})$ is embedded in
  $(\leb{\infty}(S^1),\|\cdot\|_{\leb{\infty}})$ and the mean of $u_h$ is constant in time we observe that
  $\Norm{u_h}_{\leb{\infty}(0,T;\leb{\infty}(S^1))}$ is bounded in
  terms of the initial (discrete) energy.
\end{remark}

\begin{remark}[Choice of discrete operators]
  While the precise choices of ``surface energy'' and dissipation
  terms (on the discrete level) were somewhat arbitrary in
  \cite{GiesselmannMakridakisPryer:2014} this is not the case here.  Our analysis heavily relies
  on the fact that $a_h^d$ is coercive on $\rV_q^m$ in order to infer
  an error estimate from the relative entropy estimate Corollary
  \ref{Gron}.  We choose the same kind of gradient operators for
  discretising the viscous term in \eqref{sds} as for the gradient in
  the continuity equation in order to simplify the estimates for the
  residual $R_v$ in Proposition \ref{ee}.  Let us finally note that
  the roles of $G^+$ and $G^-$ in \eqref{sds} could be interchanged.
\end{remark}

\begin{lemma}[Stability of the $\leb{2}$ projection]
  The $P_q$ projection is stable with respect to the dG-seminorm.
\end{lemma}
\begin{proof}
  Arguing similarly to the proof of \cite[Lem 4.6]{GLV11} we have for
  any $w \in \sobh{1}(\T{})$
  \begin{equation}
    \begin{split}
      \norm{P_q[w]}_{dG}^2
      &=
      \sum_n \qp{ \int_{I_n} (\del_x(P_q w ))^2 \d x + \frac{\jump{P_q[ w]}_n^2}{h_{n-1}+h_n}}\\
      &\leq \sum_n \qp{ \int_{I_n} (\del_x(P_q [w] - P_0[w]) )^2 \d x +\frac{ 2\jump{P_q[ w]- w}_n^2 + 2 \jump{w}_n^2  }{h_{n-1}+h_n}  }\\
      &\leq \sum_n \Bigg( \int_{I_n} h_n^{-2}\qp{P_q [w] - P_0[w] }^2 \d x + 2\int_{I_n} \frac{\qp{P_q [w] - w }^2}{(h_{n-1}+h_n)^2}  \d x 
        + 2\frac{ \jump{w}_n^2 }{h_{n-1}+h_n}\Bigg)\\
      &\leq \sum_n \qp{3 \int_{I_n} (\del_x(w ))^2 \d x + 2\frac{\jump{w}_n^2}{h_{n-1}+h_n}   } \leq 3 \norm{w}_{dG}^2,
    \end{split}
  \end{equation}
  concluding the proof.
\end{proof}

We are now in position to prove the existence of solutions to
\eqref{sds} for arbitrary long times:
\begin{lemma}[Existence and uniqueness to the discrete scheme (\ref{sds})]
  For given initial data $u_h^0, v_h^0 \in \fes_q$ the ODE system
  \eqref{sds} has a unique solution $(u_h,v_h,\tau_h) \in
  \qp{\cont{1}((0,\infty),\fes_q)}^3.$
\end{lemma}

\begin{proof}
  To some $w_h \in \fes_q$ let $\Delta_h w_h$ denote the unique
  element of $\fes_q$ satisfying
  \[ a_h^d (w_h,\Phi) = - \int_{S^1} \Phi \Delta_h w_h \d x.\] Using
  this notation we may remove $\tau_h$ from \eqref{sds} and rewrite it
  is
  \begin{equation}\label{sds:rew}
    \begin{split}
      \int_{S^1} \del_t u_h \Phi - G^-[v_h]\Phi \d x 
      &=0 
      \quad \forall \ \Phi \in \rV_q, \\
      \int_{S^1}\del_t v_h \Psi- G^+\qb{P_q[W'(u_h)] - \gamma \Delta_h u_h} \Psi
      + 
      \mu G^-[v_h]G^-[\Psi]\d x 
      &=0
      \quad \forall \ \Psi \in \rV_q.
    \end{split}
  \end{equation}
  This can be written in more abstract form as
  \begin{equation}
    \label{ode} 
    z'(t) = f(z(t)),
  \end{equation}
  with
  \begin{equation}
    z:= \begin{pmatrix}
      u_h \\ v_h
    \end{pmatrix}
    \quad 
    f(z):= \begin{pmatrix}
      G^-[z_2]\\
      G^+\qb{P_q[W'(z_1)] - \gamma \Delta_h z_1} + \mu G^+[G^-[z_2]]
    \end{pmatrix}.
  \end{equation}
  Note that $f: (\fes_q)^2 \rightarrow (\fes_q)^2$ is continuous, due
  to inverse estimates and stability of projection operators.  As
  $\fes_q$ is finite dimensional we do not need to choose a norm on
  $\fes_q.$ From Remark \ref{rem:ed}, the coercivity of $a_h^d$
  \eqref{a:ahd} and the fact that the mean value of $u_h$ does not
  change over time we infer that $z(t)$ remains in some bounded set $
  K \subset (\fes_q)^2$ (depending on the initial data) as long as a
  classical solution exists.  Note that this conclusion does not require any growth assumptions
  on $W.$ Note also that $K$ can be chosen such that for any initial data $
  z^0 \in K$ solutions remain in $K.$ For any $z \in (\fes_q)^2$ we
  have that
  \[ \D f(z): (\fes_q)^2 \rightarrow (\fes_q)^2,\]
  with
  \[ \D f(z)(\tilde z) = 
  \begin{pmatrix}
    G^-[\tilde z_2]\\
    G^+[P_q[W''(z_1)\tilde z_1] - \gamma \Delta_h \tilde z_1] + \mu G^+[G^-[\tilde z_2]]
  \end{pmatrix}.
  \]
  Thus, the regularity of $W$ implies that $\D f(z) $ is a uniformly bounded
  operator for all $z \in K.$
  Thus, Picard-Lindel\"of's theorem implies that for any initial data
  $z^0 \in K$ there is a local solution to \eqref{sds} with a minimal
  time of existence bounded uniformly from below.

  Let us now assume that initial data $z^0 \in (\fes_q)^2$ are given
  and there is a maximal finite interval of existence $[0,T_m)$ with
  $T_m < \infty$ of the associated solution.  Let $K$ be the set of
  elements in $(\fes_q)^2$ with energy smaller or equal to the energy
  of the initial data.  Then the solution can be evaluated on an
  increasing sequence of times $(t_i)_{i \in \rN}$ with
  \[ t_i < t_{i+1} < T_m ,\ z(t_i) \in K \Foreach i, \ \lim_{i
    \rightarrow \infty} t_i =T_m.\] Then, there is some $i$ such that
  the difference between $T_m$ and $t_i$ is smaller that the minimal
  time of existence of solutions for \eqref{sds} with initial data in
  $K.$ Thus, we can extend the solution on $[0,T_m)$ by the solution
  with ``initial'' data $(t_i,z(t_i))$ which is a contradiction to the
  maximality of $T_m.$
\end{proof}

\section{The discrete relative entropy framework}\label{sec:dre}
The stability analysis of (nonlinear systems of) hyperbolic
conservation laws is based on the relative entropy framework, which
transfers the knowledge about the energy dissipation inequality into
estimates for differences of solutions. This framework cannot be used
here directly as $W$, and therefore the whole energy, is not
convex. It was shown in \cite{Gie}, however, that the higher order
regularization terms in \eqref{thorder} make it possible to consider
only part of the relative entropy and thereby obtain stability
results.  In this section we will employ the fact that our
semi-discrete scheme \eqref{sds} satisfies a discrete energy
inequality, see Remark \ref{rem:ed}, in order to obtain a discrete
version of the results in \cite{Gie}.

\begin{definition}[Discrete reduced relative entropy]
  For tuples $(u_h,v_h,\tau_h)$ and $(\tilde u_h,\tilde v_h,\tilde \tau_h)
  \in \cont{0}([0,T],\rV_q)^3 $ we define the reduced relative entropy
  between them as
  \begin{multline}
    \eta_R(t) 
    :=
    \frac{1}{2} \Norm{v_h(t,\cdot) - \tilde v_h(t,\cdot)}_{\leb{2}(S^1)}^2 
    +
    \frac{\gamma}{2} a_h^d(u_h(t,\cdot) - \tilde u_h(t,\cdot),u_h(t,\cdot) - \tilde u_h(t,\cdot)) \\+ \frac{\mu}{4} \int_0^t \qp{G^-[v_h (s,\cdot)- \tilde v_h(s,\cdot)]}^2 \d s.
  \end{multline}
\end{definition}

\begin{lemma}[Discrete reduced relative entropy rate]\label{lem:drrer}
  Let $(u_h,v_h,\tau_h)$ be a solution of \eqref{sds} and let
  \[(\tilde u_h,\tilde v_h,\tilde \tau_h) \in \cont{1}([0,T),\rV_q)
  \times \cont{1}([0,T),\rV_q)\times \cont{0}([0,T),\rV_q)\] be a
  solution of the following perturbed problem
  \begin{equation}\label{sds:pert}
    \begin{split}
      \int_{S^1} \del_t \tilde u_h \Phi - G^-[\tilde v_h]\Phi \d x &=\int_{S^1} R_u \Phi \d x \quad \forall \ \Phi \in \rV_q \\
      \int_{S^1}\del_t \tilde v_h \Psi- G^+[\tilde \tau_h]\Psi  + \mu G^-[\tilde v_h] G^-[\Psi]\d x &=\int_{S^1} R_v \Psi \d x \quad \forall \ \Psi \in \rV_q \\
      \int_{S^1}\tilde \tau_h Z - W'(\tilde u_h) Z \d x - \gamma a_h^d(\tilde u_h,Z) &=\int_{S^1} R_\tau Z \d x \quad \forall \ Z \in \rV_q ,
    \end{split}
  \end{equation}S
  for some $R_u,R_v,R_\tau \in \cont{0}([0,T),\rV_q).$ Then the rate
  (of change) of the discrete reduced relative entropy satisfies
  \begin{equation}\label{eq:rdrre}
    \begin{split}
      \frac{\d}{\d t} \eta_R  
      &=
      - \frac{3}{4} \mu \int_{S^1} G^-[v_h-\tilde v_h] G^-[v_h - \tilde v_h] \d x
      \\
      &\qquad - \int_{S^1}  R_v (v_h - \tilde v_h) + R_u (\tau_h - \tilde \tau_h)
      + (W'(u_h) -W'(\tilde u_h)) G^-[v_h - \tilde v_h]\d x
      \\
      &\qquad +
      \int_{S^1}  (W'(u_h) -W'(\tilde u_h))R_u + R_\tau G^-[v_h - \tilde v_h] - R_\tau R_u\d x.
  \end{split}
\end{equation}
\end{lemma}

\begin{remark}[Impact of different residuals]
  If we consider applying Gronwall's Lemma to \eqref{eq:rdrre} we
  observe that the residual $R_u$ is more problematic than
  $R_v,R_\tau$ as it is multiplied by $\tau_h - \tau_h$ which is not
  controlled by the reduced relative entropy. While it is possible to
  replace this term using \eqref{sds}$_3$ and \eqref{sds:pert}$_3$
  this would in turn introduce a term $a_h^d(u_h - \tilde u_h,R_u)$,
  which includes derivatives of $R_u$.  Therefore, our projections in
  Section \ref{sec:ee} will be constructed such that $R_u=0$. The
  discrete relative entropy rate in this case is considered in more
  detail in the subsequent corollary.
\end{remark}

\begin{corollary}[Estimate of reduced relative entropy]\label{Gron}
  Let the conditions of Lemma \ref{lem:drrer} be satisfied with
  $R_u=0.$ Let $\tilde u_h$ be bounded in $\leb{\infty}(0,T;\sob{1}{\infty}(
  S^1))$ and satisfy
  \begin{equation}\label{eq:mean}
    \int_{S^1} u_h(0,\cdot) - \tilde u_h(0,\cdot) \d x =0.      
  \end{equation}
  Then, there exists a constant $C>0$ depending only on $\gamma,T,
  u_0, v_0, \Norm{\tilde
    u_h}_{\leb{\infty}(0,T;\sob{1}{\infty}(S^1))}$ such that for $0 \leq
  t \leq T$
  \[   \frac{\d}{\d t}  \eta_R(t)  \leq C \eta_R(t) + C \int_{S^1} R_v^2(t,\cdot) + \frac{1}{\mathfrak{h}^2}R_{\tau}^2(t,\cdot) \d x.\]
  Therefore, Gronwall's Lemma implies (for $0 \leq t \leq T$)
  \begin{equation}\label{Gron2}
    \eta_R(t) 
    \leq
    \Big( \eta_R(0)  + C\|R_v\|_{\leb{2}([0,t]\times S^1) }^2 
    +
    {C} \|\mathfrak{h}^{-1}R_\tau\|_{\leb{2}([0,t]\times S^1) }^2\Big) \exp(Ct).
  \end{equation}
\end{corollary}

\begin{proof}
  Upon using $R_u=0$, \eqref{dual} and Young's inequality on the
  assertion of Lemma \ref{lem:drrer} we obtain
  \begin{equation}\label{1303:1}
    \frac{\d}{\d t}  \eta_R 
    \leq
    \int_{S^1} R_v^2 + 2 (v_h - \tilde v_h)^2
    +
    (G^+[P_q[W'(u_h) -W'(\tilde u_h)]])^2  +(G^+[ R_\tau])^2\d x.
  \end{equation}
  Because of Lemma \ref{prop:grad1}, \eqref{1303:1} implies
  \begin{equation}\label{1303:2}
    \frac{\d}{\d t}  \eta_R  \leq   \int_{S^1} R_v^2 + 2(v_h - \tilde v_h)^2
    +
    \frac{C}{\mathfrak{h}^2} R_\tau^2\d x + \norm{P_q[W'(u_h) -W'(\tilde u_h)]}_{dG}^2.
  \end{equation}
  Using the stability of the $\leb{2}$ projection with respect to the
  dG-norm we get
  \begin{equation}
    \label{1303:3}
    \begin{split}
    \frac{\d}{\d t}  \eta_R 
    &\leq   \int_{S^1} R_v^2 + 2(v_h - \tilde v_h)^2
    +
    \frac{C}{\mathfrak{h}^2} R_\tau^2\d x + C\norm{W'(u_h) -W'(\tilde u_h)}_{dG}^2
    \\
    &\leq
    \int_{S^1} R_v^2 + 2(v_h - \tilde v_h)^2
    +
    \frac{C}{\mathfrak{h}^2} R_\tau^2\d x + C \Norm{u_h -\tilde u_h}_{dG}^2
    .
  \end{split}
\end{equation}
For the second inequality in \eqref{1303:3} we have used the fact that
\begin{equation}
  \begin{split}
    \norm{W'(u_h) -W'(\tilde u_h)}_{dG}^2
    &
    \leq
    \sum_n \bigg(
    \Norm{\qp{W''(u_h) -W''(\tilde u_h)} \del_{x} \tilde u_h }_{\leb{2}(I_n)}^2 
    \\
    &\qquad +
    \Norm{W''(\tilde u_h)\qp{ \del_{x} \tilde u_h - \del_x u_h} }_{\leb{2}(I_n)}^2
    +
    \frac{2\jump{ W'(u_h) - W'(\tilde u_h)}_n^2}{h_{n-1}+h_n}
    \bigg)
    \\
    &\leq 
    \sum_n \bigg( \norm{\tilde u_h}_{\sob{1}{\infty}}^2 \Norm{W''(u_h) -W''(\tilde u_h)}_{\leb{2}(I_n)}^2
    \\
    &\qquad + \Norm{W''(\tilde u_h)\qp{ \del_{x} \tilde u_h - \del_x u_h} }_{\leb{2}(I_n)}^2
    + \frac{2\jump{ W'(u_h) - W'(\tilde u_h)}_n^2}{h_{n-1}+h_n} \bigg)\\
    &\leq 
    C\sum_n \bigg( \norm{\tilde u_h}_{\sob{1}{\infty}}^2 \Norm{u_h -\tilde u_h}_{\leb{2}(I_n)}^2+ \Norm{\del_{x} \tilde u_h - \del_x u_h}_{\leb{2}(I_n)}^2
    + \frac{2\jump{ u_h - \tilde u_h}_n^2}{h_{n-1}+h_n}\bigg) ,
  \end{split}
\end{equation}
because $\Norm{W}_{\sob{3}{\infty}[-M,M]}$ is bounded for
\[ M:= \max\{\Norm{\tilde u_h}_{\leb{\infty}(0,T;\leb{\infty}(S^1))},
\Norm{u_h}_{\leb{\infty}(0,T;\leb{\infty}(S^1))}\}.\] The assertion of
the Lemma follows from \eqref{1303:3} as
\[ \Norm{u_h -\tilde u_h}_{dG}^2 \leq C \norm{u_h -\tilde u_h}_{dG}^2
\leq C a_h^d (u_h -\tilde u_h, u_h -\tilde u_h)\] due to
\eqref{eq:mean}.
\end{proof}

\begin{remark}[Parameter dependence]
  Note that the constant $M$ in the proof of Corollary \ref{Gron}
  depends on $\gamma$ which induces a subtle dependence of $C$ in
  \eqref{Gron2} on $\gamma$ which is intertwined with the growth
  behaviour of $W$ and its derivatives.  There is an additional
  $\gamma$ dependence of $C$ which enters when
  \[ \Norm{u_h - \tilde u_h}_{\leb{2}(S^1)}^2 + \Norm{v_h - \tilde
    v_h}_{\leb{2}(S^1)}^2\] is estimated by $C \eta_R.$ This leads to
  a subtle dependence of all the constants $C$ in the subsequent
  results on $\gamma$ and $C$ behaves like $1/\gamma$ at best.
\end{remark}

In case the reader takes special interest in the sharp interface case
$\gamma \rightarrow 0$ we like to state the following result which
shows that the previous estimate can also be obtained in a
uniform-in-$\gamma$ version. However, in that case, the stability
constant sensitively depends on $\mu$.

\begin{corollary}[Estimate of modified relative entropy]\label{cor:Gron3}
  Let the assumptions of Lemma \ref{lem:drrer} be satisfied with
  $R_u=0.$ Let $|W''|$ be uniformly bounded.  Then, there exists a
  constant $C>0$ depending only on $\mu,T, u_0, v_0,
  \Norm{W''}_{\leb{\infty}(\rR)}$ such that for $0 \leq t \leq T$
  \[ \eta_M (t):= \frac{1}{2}\Norm{u_h(t,\cdot) - \tilde
    u_h(t,\cdot)}_{\leb{2}(S^1)}^2 + \eta_R(t)\] satisfies
  \begin{equation}\label{Gron3a}   \frac{\d}{\d t}  \eta_M(t) \leq
    C  \eta_M(t)  + C \int_{S^1} R_v^2(t,\cdot) + R_{\tau}^2(t,\cdot) \d x.\end{equation}
  Therefore, Gronwall's Lemma implies (for $0 \leq t \leq T$)
  \begin{equation}\label{Gron3}
    \eta_M(t) \leq 
    C\Big(  \eta_M(0)  + \|R_v\|_{\leb{2}([0,t]\times S^1) }^2 +  \|R_\tau\|_{\leb{2}([0,t]\times S^1) }^2\Big) \exp(Ct).
  \end{equation}
\end{corollary}
\begin{proof}
  Starting from \eqref{eq:rdrre} with $R_u=0$ and $|W''|$ uniformly
  bounded we find
  \begin{multline}\label{eq:25041}
    \frac{\d}{\d t} \eta_R  \leq \int_{S^1} -\frac{3}{4}\mu  |G^-[v_h-\tilde v_h]|^2 
    + R_v^2 +(v_h - \tilde v_h)^2 
    + \frac{C}{\mu} (u_h -\tilde u_h)^2  \d x\\
    +\int_{S^1}  \frac{\mu}{4}|G^-[v_h - \tilde v_h]|^2 + \frac{1}{\mu}R_\tau^2 + \frac{\mu}{4} |G^-[v_h - \tilde v_h]|^2\d x.
  \end{multline}
  In addition, because of \eqref{sds}$_1$ and \eqref{sds:pert}$_1$, it
  holds
  \begin{equation}
    \label{eq:25042}
    \begin{split}
      \frac{\d}{\d t} 
      \qp{\frac{1}{2}
        \Norm{u_h - \tilde u_h}_{\leb{2}(S^1)}^2
      }
      &=
      \int_{S^1} (u_h - \tilde u_h)\del_t (u_h - \tilde u_h)\d x
      \\
      &= \int_{S^1} (u_h - \tilde u_h)G^-[v_h - \tilde v_h]  \d x
      \\
      &\leq \int_{S^1}  \frac{1}{\mu} (u_h - \tilde u_h)^2 + \frac{\mu}{4} |G^-[v_h - \tilde v_h]|^2 \d x.
    \end{split}
  \end{equation}
  Adding \eqref{eq:25041} and \eqref{eq:25042} we obtain
  \begin{equation}
    \begin{split}
      \frac{\d}{\d t} \eta_M 
      &\leq \int_{S^1} 
      R_v^2 +(v_h - \tilde v_h)^2 
      + 
      \frac{1}{\mu}R_\tau^2 + \frac{C}{\mu} (u_h - \tilde u_h)^2   \d x
      \\
      &\leq C \eta_M + \int_{S^1} 
      R_v^2  
      +    \frac{1}{\mu}R_\tau^2 \d x ,
    \end{split}
\end{equation}
which proves \eqref{Gron3a} and \eqref{Gron3} follows by Gronwall's
inequality.
\end{proof}

\begin{remark}[Parameter dependence of the constant in (\ref{Gron3})]
  Note that the constant $C$ in \eqref{Gron3} scales like $1/\mu$ for
  $\mu \rightarrow 0.$
\end{remark}

\begin{proof}{of Lemma \ref{lem:drrer}.}
  A direct computation shows
  \begin{equation}\label{1303:4}
    \frac{\d}{\d t}  \eta_R  =\int_{S^1} (v_h - \tilde v_h) (\del_t v_h - \del_t \tilde v_h) + \frac{\mu}{4} \qp{G^-[v_h - \tilde v_h]}^2\d x + \gamma a_h^d(u_h - \tilde u_h, \del_t u_h - \del_t \tilde u_h). 
  \end{equation}
  Using $Z=\del_t (u_h - \tilde u_h)$ and $\Psi =v_h - \tilde v_h$ in
  \eqref{sds} and \eqref{sds:pert} we infer from \eqref{1303:4} that
  \begin{equation}\label{1303:5}
    \begin{split}
      \frac{\d}{\d t} 
      \eta_R 
      &=\int_{S^1} (v_h - \tilde v_h) G^+[ \tau_h - \tilde \tau_h] -R_v (v_h - \tilde v_h) \d x 
      \\
      &\qquad +
      \int_{S^1} (\tau_h - \tilde \tau_h) (\del_t u_h - \del_t \tilde u_h)
      - (W'(u_h) -W'(\tilde u_h))( \del_t u_h - \del_t \tilde u_h) \d x
      \\
      &\qquad
      + \int_{S^1}
      R_\tau (\del_t u_h - \del_t \tilde u_h)- \frac{3}{4}\mu G^-[v_h - \tilde v_h] G^-[v_h - \tilde v_h]\d x . 
    \end{split}
  \end{equation}
  Using $\Phi=(\tau_h- \tilde \tau_h)$ as a test function in
  \eqref{sds} and \eqref{sds:pert} and employing \eqref{dual} we
  obtain
  \begin{equation}\label{1303:6}
    \begin{split}
      \frac{\d}{\d t}  \eta_R 
      &=
      \int_{S^1}  -R_v (v_h - \tilde v_h) -R_u(\tau_h - \tilde \tau_h)
      - (W'(u_h) -W'(\tilde u_h))( \del_t u_h - \del_t \tilde u_h) \d x
      \\
      &\qquad +
      \int_{S^1}
      R_\tau (\del_t u_h - \del_t \tilde u_h)-\frac{3}{4} \mu G^-[v_h - \tilde v_h] G^-[v_h - \tilde v_h]\d x . 
    \end{split}
  \end{equation}
  As $( \del_t u_h - \del_t \tilde u_h) \in \rV_q$ for each $0 \leq t
  \leq T$ we may replace $(W'(u_h) -W'(\tilde u_h))$ by its $\leb{2}$ projection $P_q [W'(u_h) -W'(\tilde u_h)]$ in \eqref{1303:4}.  Upon
  using $\Phi = P_q [W'(u_h) -W'(\tilde u_h)]-R_\tau$ in \eqref{sds}
  and \eqref{sds:pert} we obtain the assertion of the Lemma from
  \eqref{1303:6}.
\end{proof}

\section{Projections and perturbed equations}\label{sec:proj}
Let $(u,v)$ be a strong solution of \eqref{thorder}, see Proposition
\ref{prop:thorder}. We aim at determining projections of $(u,v)$ and
$\tau:= W'(u) - \gamma \del_{xx}u $ so that these projections form a
perturbed solution of \eqref{sds} such that there is no residual in
the first equation and the residuals in the other equations are of
optimal order.

It is important to appropriately account for the highest
order derivative, as such, we project $u$ by the Riesz projection,
defined in \eqref{Riesz:def}.  Let us note that due to the linearity
of the definition of the Riesz projection we have
\begin{equation}
  \del_t \ritz [u]= \ritz [\del_t u] = \ritz [\del_x v].
\end{equation}

Since our aim is ensuring that the projections satisfy \eqref{sds}$_1$
exactly, this already determines the discrete gradient of the
projection of $v$.  Before we can focus on the projection of $v$ we
need to investigate the kernel and range of the gradient operators
$G^\pm$.  To this end we need to introduce some notation: By $l_k\in
\poly{k}(-1,1)$ we denote the $k$-th Legendre polynomial on $(-1,1)$
and by $l_k^n$ its transformation to the interval $I_n$, \ie
\begin{equation}
  l^n_k(x)
  =
  l_k\qp{2\qp{\frac{x-x_n}{h_n}} -1}
  .
\end{equation}
Let us gather the key properties of the Legendre polynomials which we
will employ in the sequel:
\begin{proposition}[Properties of the Legendre polynomials \cite{AW05}]\label{pro:Legendre}
  The transformed Legendre polynomials $l_k^n$ have the following properties
  \begin{gather}
    \label{Legendre1}
    (-1)^k l^n_k(x_n)= l^n_k(x_{n+1})=1,
    \\
    \label{Legendre3} 
  0\leq  \int_{I_n} l^n_{k'}(x)l^n_k(x) \d x 
     =
    \frac{h_n}{2k+1}\delta_{kk'} \leq h_n,\\
    \label{Legendre4}
     \Norm{ l^n_k}_{\leb{\infty}(I_n)}  \leq 1.
  \end{gather}
\end{proposition}

Let us point out the following convention in our notation for the subsequent calculations:
Superscripts will usually refer to the element/interval/vertex under
consideration while subscripts refer to the polynomial degree.  The
only exception is $h_n$ denoting the length of the $n$-th interval.

\begin{lemma}[The kernel of $G^\pm$]\label{lem:kernel}
  The kernel of each of the operators $G^\pm: \rV_q \rightarrow \rV_q$
  defined in \eqref{dgrad} is one dimensional and consists of the
  functions which are constant everywhere.  The range of $G^\pm$ is
  $\rV_q^m.$
\end{lemma}
\begin{proof}
  We will give the proof for the kernel of $G^+,$ the modifications
  for $G^-$ are straightforward.  Consider $\Phi \in \rV_q$ with
  $G^+[\Phi]=0.$ Let us fix some $n$ and define $\Psi \in \fes_q$ by
 \[ \Psi(x) := \left\{
   \begin{array}{ccc}
     l_q^n(x) &:& x \in I_n \\ 0 &:& x \not \in I_n
   \end{array}\right.\]
 we find, as $\del_x(\Phi|_{I_n}) \in \rP^{q-1}(I_n),$
 \[ 0 = \int_{S^1} G^+[\Phi] \Psi\d x = \sum_n \Big( \int_{I_n}
 \del_x\Phi \Psi \d x - \Psi(x_n^+) \jump{\Phi}_{n} \Big) = (-1)^{q+1}
 \jump{\Phi}_{n}.\]
 As $n$ was arbitrary we obtain that $\Phi$ is continuous.  The
 continuity of $\Phi$ implies
 \[ 0 = \int_{S^1} G^+[\Phi]G^+[\Phi] \d x = \sum_n \int_{I_n}
 (\del_x\Phi )^2 \d x.\] Therefore, $\Phi$ is continuous and constant
 in each interval. Thus, $\Phi$ is globally constant and the assertion
 for the kernel is proven.  We infer from the result for the kernel
 that the range of $G^\pm$ has codimension $1$.  The proof is
 concluded by the observation
 \[ \int_{S^1 } G^\pm [\Phi] \d x = \sum_{n} \qp{ \int_{I_n} \del_x
   \Phi \d x - \jump{\Phi}_n }=0 \Foreach \Phi \in \rV_q,\] which
 implies that the range of $G^\pm$ is a subset of $\rV_q^m.$
\end{proof}

\begin{remark}[Properties of one sided gradients]
  The properties of $G^\pm$ asserted in Lemma \ref{lem:kernel}
  distinguish them from the ``central'' discrete gradients used in
  \cite{GiesselmannMakridakisPryer:2014} which may have $2$-dimensional kernels.
\end{remark}

Our next aim is to show the following discrete Poincar\'{e} inequality:
\begin{lemma}[Discrete Poincar\'{e} inequality]\label{lem:dpi}
  There exists a constant $C>0$ independent of $h$ such that
  \[ \| \Phi \|_{\leb{2}(S^1)} \leq C  \| G^-[\Phi] \|_{\leb{2}(S^1)} \Foreach \Phi \in \fes_q^m.\]
\end{lemma}
\begin{proof}
  For each interval $I_n$ let $D_n$ denote the map
  \[ \operatorname{span}\{ l_1^n, \dots , l_q^n\} \rightarrow
  \operatorname{span}\{ l_0^n, \dots , l_{q-1}^n\}, \qquad \Phi
  \mapsto \del_x\Phi. \] Since $\ker{D_n}$
  is trivial, as it consists of functions which are constant and
  orthogonal to constant functions, we have that $D_n$ is invertible.  Comparing $D_n$ to the analogous
  map on $(-1,1),$ instead of $I_n,$ we obtain that $\|D_n^{-1}\|_2
  =\mathcal{O}( h_n)$, where $\Norm{\cdot}_2$ denotes the Euclidean matrix norm.  Let us now write the functions under
  consideration as linear combinations of transformed Legendre
  polynomials in each interval
  \begin{equation}
    G^-[\Phi]|_{I_n}(x) = \sum_{r=0}^q g^n_r l^n_r(x) , \quad \Phi|_{I_n}(x) = \sum_{r=0}^q a^n_r l^n_r(x) , \quad \del_x(\Phi|_{I_n})(x) = \sum_{r=0}^{q-1} b^n_r l^n_r(x),
  \end{equation}
  with real numbers $(g^n_r)_{r=0,\dots,q}^{n=0,\dots,N-1}$,
  $(a^n_r)_{r=0,\dots,q}^{n=0,\dots,N-1}$,
  $(b^n_r)_{r=0,\dots,q-1}^{n=0,\dots,N-1}.$ Let $\chi^n$ denote the
  characteristic function of $I_n$.  Then we have by definition of
  $G^-$
 \begin{equation}\label{dpi:1a}
   \int_{S^1} G^-[\Phi] (l^n_r -l^n_q ) \chi^n \d x = \int_{S^1} \del_x\Phi l_r^n\chi^n \d x \Foreach r=0,\dots,q-1  ,
 \end{equation}
 as $\del_x \Phi$ is orthogonal to $l^n_q$ and $(l^n_r
 -l^n_q)(x_{n+1}^-)=0,$ and
 \begin{equation}\label{dpi:1b}
   \int_{S^1} G^-[\Phi]  l^n_q  \chi^n \d x = - \jump{\Phi}_{n+1}
 \end{equation}
 because $l^n_q(x_{n+1}^-)=1.$ This implies
 \begin{equation}\label{dpi:2}
   \frac{g^n_r}{2r +1 } - \frac{g^n_q}{2q +1 }= \frac{b^n_r}{2r +1 } \ \forall r=0,\dots,q-1 \quad \text{and} \quad  \frac{g^n_q h_n }{2q +1 } = \sum_{r=0}^q (-1)^r a^{n+1}_r-  \sum_{r=0}^q a^n_r .
 \end{equation}
 From \eqref{dpi:2}$_1$ we infer 
 \begin{equation}\label{dpi:3}
   | b_r^n|   \leq   | g_r^n| +   | g_q^n|.
 \end{equation}
 For $\vec a^n=\Transpose{(a_1^n,\dots,a_q^n)},$ $\vec g^n=
 \Transpose{(g_0^n,\dots,g_{q}^n)},$ and $\vec
 b^n=\Transpose{(b_0^n,\dots,b_{q-1}^n)}$ we have $\Norm{\vec b^n}
 \leq C \Norm{\vec g^n}$ and $\vec b^n = D_n \vec a^n$ such that
 \begin{equation}\label{dpi:3b} \| \vec a^n\| \leq C  h_n \|\vec g^n\|,\end{equation}
 as $\|D_n^{-1}\|_2 =\mathcal{O}(h_n)$.
 
 From \eqref{dpi:2}$_2$ we infer 
 \begin{equation}\label{dpi:4}
   a^n_0 - a^{n+1}_0 = -\frac{g^n_q h_n }{2r +1 } -\sum_{r=1}^q a^n_r +  \sum_{r=1}^q (-1)^r a^{n+1}_r =: c^n
 \end{equation}
 with $c^n = \mathcal{O}(h_n (\Norm{\vec g^n} + \Norm{\vec g^{n+1}}))$ for each $n$ due to \eqref{dpi:3b}.
 As $\Phi \in \fes_q^m$ we have $\sum_{n=0}^{N-1} a_0^n =0.$ Therefore,
 $\tilde{ \vec a} =\Transpose{(a^0_0,\dots, a^{N-1}_0)}$ and $\vec c =\Transpose{(c^0, \dots, c^{N-1})}$ satisfy
 \begin{equation}
   \begin{split}
     \| \tilde {\vec a}\|_2^2 
     &=
     \sum_{n=0}^{N-1} (a^n_0)^2
     = \sum_{n=0}^{N-1} \qp{a^n_0 - \frac{1}{N} \sum_{j=0}^{N-1} a^j_0}^2\\
     &= 
     \sum_{n=0}^{N-1}  \qp{\frac{1}{N} \sum_{j=0}^{N-1} a^n_0 - a^j_0}^2
     \leq \sum_{n=0}^{N-1} \sum_{j=0}^{N-1} \frac{1}{N} \qp{a^n_0 - a^j_0}^2\\
     &\leq
     \sum_{n=0}^{N-1} \sum_{j=0}^{N-1} \frac{1}{N} \qp{\sum_{k=0}^{N-1} | c^k|}^2
     \leq \sum_{n=0}^{N-1} \sum_{j=0}^{N-1} \sum_{k=0}^{N-1} | c^k|^2 = N^2 \| \vec c\|_2^2,
   \end{split}
 \end{equation}
 where we used Jensen's inequality, the definition of $c^n$ and Cauchy-Schwarz inequality.
 Combining the preceding estimates we conclude
 \begin{equation}
   \begin{split}
     \Norm{\Phi}_{\leb{2}(I)}^2 &\leq \sum_{n=0}^{N-1} \sum_{r=0}^q h_n |a_r^n |^2\\
     &\leq  h \sum_{n=0}^{N-1}|a_0^n |^2 + \sum_{n=0}^{N-1} \sum_{r=1}^q h_n |a_r^n |^2 \\
     &\leq h\qp{ \sum_{n=0}^{N-1}|a_0^n |^2 + \sum_{n=0}^{N-1} \Norm{\vec a^n}^2}\\
     & \leq C h N^2 \sum_{n=0}^{N-1} |c^n|^2 + C \sum_{n=0}^{N-1} h^3 \| \vec g^n\|^2\\
     &\leq C h N^2 \sum_{n=0}^{N-1} h^2\|\vec g^n\|^2 + C \sum_{n=0}^{N-1} h^3 \| \vec g^n\|^2\\
     & \leq C h \sum_{n=0}^{N-1} \sum_{r=0}^q |g_r^n |^2 \leq C \Norm{G^-[\Phi]}_{\leb{2}(I)}^2 ,
   \end{split}
 \end{equation}
where we have used that $hN$ is bounded.
\end{proof}

\begin{definition}[Projection $Q$]\label{def:proj1}
  For $q \in \rN$ we define $ S_q^\pm : \cont{0}(S^1) \rightarrow
  \rV_q $ by
  \[  S_q^\pm[w](x_n^\pm)= w(x_n), \quad \int_{S^1} (S_q^\pm[w] - w)  \Phi\d x =0 \quad \forall \ \Phi \in \rV_{q-1}.\]
  We also define
  $Q:  \cont{1}(S^1)  \rightarrow \rV_q $ by
  \begin{equation}\label{qdef}
    G^-[Q[w]] = \ritz [\del_xw] \quad \text{and} \quad  \int_{S^1} Q[w] - w \d x =0.
  \end{equation}
  Note that $Q[w]$ is well-defined by \eqref{qdef} due to Lemma \ref{lem:kernel} and the fact that $\int_{S^1} \ritz [\del_xw] \d x=\int_{S^1} \del_x w\d x =0$ as $w$ is periodic. 
\end{definition}

\begin{lemma}[Properties of the projection operator $Q$]\label{lem:proj1}
  The projection operators from Definition \ref{def:proj1} satisfy the following estimates: There exists a $C>0,$ independent of $h,$ such that
  for every $w \in \sobh{q+3}(S^1)$
  \begin{equation}\label{eq:proj1}
    \begin{split}
      \| S_q^\pm[w] - w\|_{\leb{2}(S^1)} &= Ch^{q+1} \Norm{w}_{\cont{q+1}(S^1)}\\
      \Norm{ G^-\qb{Q[w] - S_{q}^+[w]}}_{\leb{2}(S^1)} &= Ch^{q+1} \Norm{w}_{\sobh{q+3}(S^1)}\\
      \| Q[w] - S_{q}^+[w]\|_{\leb{2}(S^1)} &= Ch^{q+1} \Norm{w}_{\sobh{q+3}(S^1)}.
    \end{split}
  \end{equation}
  
\end{lemma}
\begin{proof}
  The first assertion follows from the fact that $S_q^\pm$ is exact for functions in $\fes_q.$
  We obtain the second assertion as follows: Let $\rU:= \{ \Psi \in \fes_q : \Norm{\Psi}_{\leb{2}(S^1)}=1\},$ then
  \begin{equation}
    \begin{split}
      \Norm{G^-\qb{Q[w] - S^+_{q}[w]}}_{\leb{2}(S^1)} &= \sup_{\Psi \in \rU} \int_{S^1}  \qp{G^-\qb{Q[w] - S_{q}^+[w]}}\Psi\d x\\
      &= \sup_{\Psi \in \rU} \qp{\int_{S^1} \ritz [\del_x w] \Psi + S_q^+[w] G^+[\Psi]\d x }\\
      &= \sup_{\Psi \in \rU} \qp{\int_{S^1} \ritz [\del_x w] \Psi + S_q^+[w]\del_x \Psi \d x - \sum_n S_q^+[w](x_n^+) \jump{\Psi}_n }\\
      &= \sup_{\Psi \in \rU} \qp{\int_{S^1} \ritz [\del_x w] \Psi + w\del_x \Psi \d x - \sum_n w(x_n) \jump{\Psi}_n }\\
      &= \sup_{\Psi \in \rU} \int_{S^1} \ritz [\del_x w] \Psi - \del_x w \Psi \d x \\
      & \leq \Norm{\ritz [\del_x w] - P_q[\del_x w]}_{\leb{2}(S^1)} 
      \\
      &\leq Ch^{q+1} \Norm{w}_{\sobh{q+3}(S^1)}
    \end{split}
  \end{equation}
  because of the properties of $\ritz,$ see \eqref{Riesz:est}, $Q$,
  \eqref{dual} and $P_q$ as $\cont{q+2}(S^1) \subset \sobh{q+3}(S^1).$
  The third assertion is a consequence of the second and Lemma
  \ref{lem:dpi}.
\end{proof}

\begin{definition}[Projection $R$]\label{def:proj2}
  Let $\tau \in \cont{0}([0,T],\sobh{1}(S^1))$ and $u \in
  \cont{0}([0,T],\sobh{3}(S^1))$ be related by $\tau = W'(u)-\gamma
  \del_{xx} u.$ Then, the projection $R[\tau] \in
  \cont{0}([0,T],\rV_q)$ is defined by
  \[ \int_{S^1} R[ \tau] \Psi \d x = \int_{S^1} W'(u) \Psi \d x
  -\gamma a_h^d(\ritz [u],\Psi) \Foreach \Psi \in \fes_q.\]
\end{definition}

\begin{lemma}[Perturbed equations]
  Let $(u,v)$ be a strong solution of \eqref{thorder} and $\tau:=W'(u)-\gamma \del_{xx} u.$ Then, the projections $(\ritz [u],Q[v],R[\tau])$ satisfy
  \begin{equation}\label{sds:proj}
    \begin{split}
      \int_{S^1} \del_t \ritz [u] \Phi - G^-\qb{Q[v]}\Phi \d x &=0 \quad \forall \ \Phi \in \rV_q \\
      \int_{S^1}\del_t Q[v] \Psi- G^+\qb{R[\tau]}\Psi  + \mu G^-\qb{Q[v]} G^-[\Psi]\d x &=\int_{S^1} R_v \Psi \d x \quad \forall \ \Psi \in \rV_q \\
      \int_{S^1} R[\tau] Z - W'(\ritz [u]) Z \d x - \gamma a_h^d(\ritz [u],Z) &=\int_{S^1} R_\tau Z \d x \quad \forall \ Z \in \rV_q ,
    \end{split}
  \end{equation}
with 
\begin{equation}
 \begin{split}
    R_\tau &:= P_q[ W'(u) - W'(\ritz [u])],\\
    R_v &:= -P_q[\del_t ( v - Q[v])] + P_q[\del_x \tau] - G^+[R[\tau]] + \mu P_q[\del_{xx} v] - \mu G^+ [G^-[Q[v]]].
 \end{split}
\end{equation}
\end{lemma}

\begin{proof}
  The first equation in \eqref{sds:proj} is a direct consequence of
  the definition of $Q[v]$ in Definition \ref{def:proj1}.  The second
  equation in \eqref{sds:proj} follows from
  \begin{equation}
    \int_{S^1}\del_t v \Psi- \del_x \tau \Psi  - \mu \del_{xx}v \Psi\d x =0 \quad \forall \ \Psi \in \rV_q 
  \end{equation}
  and the duality \eqref{dual}.  The third equation follows from the
  definition of $R[\tau] $ in Definition \ref{def:proj2}.
\end{proof}

\begin{lemma}[Coercivity of $G^-$]\label{coercgm}
 There exists a constant $C>0$ only depending on $q$ such that for every $w \in \fes_q$
 \[ \norm{w}_{\operatorname{dG}} \leq C \Norm{G^-[w]}_{\leb{2}(S^1)}.\]
\end{lemma}
\begin{proof}
 Let us use 
 \[ \Psi|_{I_n}= \del_x w|_{I_n} - (-1)^q \qp{\frac{\jump{w}_{n+1}}{h_n+h_{n+1}} + \del_x w(x_{n+1}^-) }l^n_q\]
 in \eqref{dgrad}. Upon noting $\del_x w|_{I_n} \perp l^n_q$ and $\Psi(x_{n+1}^-) = \frac{\jump{w}_{n+1}}{h_n+h_{n+1}}$ we obtain
 \begin{equation}\label{bsp} \int_{S^1} G^-[w] \Psi\d x = \norm{w}_{\operatorname{dG}}^2.\end{equation}
 It remains to determine a bound for $\Norm{\Psi}_{\leb{2}}$.
 Let $\{ y_k\}_{k=0}^{q}$ denote Gauss-Radau points on $[-1,1]$ and $\{ y_k^n\}_{k=0}^{q}$ their image under the map
 \[ \kappa \mapsto \frac{x_n + x_{n+1}}{2} + \kappa \frac{x_{n+1}-x_n}{2}\]
 such that $y^n_0 =x_{n+1}$.
 By $\omega_k$ we denote the weights of Gauss-Radau quadrature.
 Due to the exactness of Gauss-Radau quadrature for polynomials of degree $2q$ and the properties of Legendre polynomials, see Proposition \ref{pro:Legendre}, we find
 \begin{equation}
   \begin{split}
     \Norm{\Psi}_{\leb{2}(I_n)}^2 
     &\leq
     2 \Norm{\del_x w|_{I_n} - (-1)^q \del_x w(x_{n+1}^-) l^n_q}_{\leb{2}(I_n)}  +2  h_n \qp{\frac{\jump{w}_{n+1}}{h_n+h_{n+1}}}^2
     \\
     &\leq 2 \sum_{k=1}^{q}h_n \omega_k (\del_x w(y^n_k) + \del_x w(y^n_0))^2+2  h_n \qp{\frac{\jump{w}_{n+1}}{h_n+h_{n+1}}}^2\\
     &\leq 4 \frac{\sum_{k=1}^{q}\omega_k }{\omega_0} h_n \sum_{k=1}^{q} (\del_x w(y^n_k))^2+ 2  h_n \qp{\frac{\jump{w}_{n+1}}{h_n+h_{n+1}}}^2\\
     &\leq 4 \frac{\sum_{k=1}^{q} \omega_k}{\omega_0} \Norm{\del_x w|_{I_n}}_{\leb{2}(I_n)} + +2  \frac{\qp{\jump{w}_{n+1}}^2}{h_n+h_{n+1}}.
 \end{split}
 \end{equation}
Summing over $n$ implies that 
\begin{equation}\label{boundpsi} \Norm{\Psi}_{\leb{2}}^2 \leq C(q) \norm{w}_{\operatorname{dG}}^2.\end{equation}
Combining \eqref{bsp} and \eqref{boundpsi} gives the desired result, as 
\[  \int_{S^1} G^-[w] \Psi\d x \leq \Norm{G^-[w]}_{\leb{2}}\Norm{\Psi}_{\leb{2}}.\]
\end{proof}

\section{Main result}\label{sec:ee}

This section is devoted to the proof of the main result of this work, which reads as follows:
\begin{theorem}[Reduced relative entropy error estimate]\label{mt}
 Let the exact solution $(u,v)$ of \eqref{thorder} satisfy 
 \begin{equation}
   \begin{split}
     u &\in \cont{1}((0,T) , \sobh{q+2}(S^1)) \cap \cont{0}([0,T] , \cont{q+3}(S^1)) \\
     v &\in \cont{1}((0,T) , \cont{q+2}(S^1)) \cap \cont{0}([0,T] , \cont{q+3}(S^1))
   \end{split}
 \end{equation}
 and let $W \in \cont{q+3}(\rR,[0,\infty)).$
 Then there exists $C>0$ independent of $h,$ but depending on $q,T,\gamma, \Norm{u}_{\leb{\infty}(0,T; \sob{1}{\infty}(S^1))}$ such that
 \begin{equation}
   \begin{split}
   \sup_{0 \leq t \leq T} \bigg(&\Norm{u_h(t,\cdot) - u(t,\cdot)}_{\operatorname{dG}} + \Norm{v_h(t,\cdot) - v(t,\cdot)}_{\leb{2}(S^1)}\bigg) 
   + \qp{\mu \int_0^T \norm{v_h(s,\cdot) - v(s,\cdot)}_{\operatorname{dG}}^2 \d s }^{1/2} \\
   &\leq C h^q\bigg( \Norm{u}_{\leb{\infty}(0,T; \cont{q+3}(S^1))} + \Norm{v}_{\leb{\infty}(0,T; \cont{q+3}(S^1))}+\Norm{\del_t v}_{\leb{\infty}(0,T; \cont{q+2}(S^1))}\bigg) .
 \end{split}
 \end{equation}
\end{theorem}

Theorem \ref{mt} is a direct consequence of the subsequent proposition,
the estimates \eqref{Riesz:est}$_1$ and \eqref{eq:proj1} and Lemma \ref{coercgm}.
\begin{proposition}[Discrete stability estimate]\label{ee}
  Under the assumptions of Theorem \ref{mt}
  there exists $C>0$ independent of $h,$ but depending on $q,T,\gamma, \Norm{u}_{\leb{\infty}(0,T; \sob{1}{\infty}(S^1))}$ such that
  \begin{equation}
    \begin{split}
      \sup_{0 \leq t \leq T} \bigg(&\Norm{u_h(t,\cdot) - \ritz [u(t,\cdot)]}_{\operatorname{dG}} 
      +
      \Norm{v_h(t,\cdot) - Q[v(t,\cdot)]}_{\leb{2}(S^1)}\bigg)
      +
      \qp{\mu \int_0^T\!\! \norm{v_h(s,\cdot) - Q[v(s,\cdot)]}_{\operatorname{dG}}^2 \d s }^{1/2}
      \\
      &\leq 
      C h^q \bigg( \Norm{u}_{\leb{\infty}(0,T; \cont{q+3}(S^1))} + \Norm{v}_{\leb{\infty}(0,T; \cont{q+3}(S^1))} +\Norm{\del_t v}_{\leb{\infty}(0,T; \cont{q+2}(S^1))} \bigg) .
    \end{split}
  \end{equation}
  
\end{proposition}
\begin{proof}
  As the subsequent estimates are uniform in time (on $[0,T]$) we omit
  the time dependency.  In order to see that Corollary \ref{Gron} can
  be applied to \eqref{sds:proj} we need $\ritz [u]$ to be bounded in
  $\leb{\infty}(0,T;\sob{1}{\infty}(S^1)).$ This follows from
  \eqref{Riesz:est} and our assumptions on $u.$ In particular, we may
  use the fact that $\Norm{W}_{\sob{3}{\infty}}$ is bounded on
  $[-M,M]$ with $M:= \max\{\Norm{
    \ritz [u]}_{\leb{\infty}},\Norm{u_h}_{\leb{\infty}}\}.$
  
  As we can apply Corollary \ref{Gron} and Lemma \ref{coercgm} it remains to estimate
  $\eta_R(0)$, $\Norm{R_v}_{\leb{2}([0,T]\times S^1)} $ and
  $\Norm{R_\tau}_{\leb{2}([0,T]\times S^1)} $.  It
  holds
 \begin{equation}
   \label{ee1}
   \begin{split}
     \eta_R(0) &\leq \Norm{u_h(0,\cdot) - \ritz [u(0,\cdot)]}_{\operatorname{dG}} + \Norm{v_h(0,\cdot) - Q[v(0,\cdot)]}_{\leb{2}(S^1)}\\
     &\leq C h^{q+1}
     \qp{ \Norm{u_0}_{ \sobh{q+2}(S^1)} + \Norm{v_0}_{\cont{q+2}(S^1)}}
 \end{split}
 \end{equation}
 by the properties of $P_q, \ritz, Q$ and $\cont{q+2}(S^1) \subset \sobh{q+2}(S^1)\subset \cont{q+1}(S^1).$

 As $|W''|$ is bounded on the interval of interest
 \begin{equation}
   \label{ee2} 
   \Norm{R_\tau}_{\leb{2}(S^1)} \leq C \Norm{u - \ritz [u]}_{\leb{2}(S^1)} \leq C h^{q+1} \Norm{u}_{\sobh{q+1}(S^1)}.
 \end{equation}
 To estimate $R_v$ we decompose it as $R_v = -R_v^1 + R_v^2+R_v^3$ with
 \begin{equation}
   \begin{split}
     R_v^1&:= P_q[\del_t ( v - Q[v])],\\
     R_v^2&:=  P_q[\del_x \tau] - G^+[R[\tau]], \\
     R_v^3&:= \mu P_q[\del_{xx} v] - \mu G^+ [G^-[Q[v]]].
   \end{split}
 \end{equation}
 The estimate $\Norm{R_v^1}_{\leb{2}(S^1)}\leq C h^{q+1} \Norm{\del_t
   v}_{\cont{q+2}(S^1)} $ follows from $\del_t Q[v]=Q[\del_t v]$,
 \eqref{eq:proj1}$_3$, the stability of $P_q,$ and our assumptions on
 $v.$ Before we consider $R_v^2$ let us recall $\rU := \{\Psi \in
 \rV_q : \Norm{\Psi}_{\leb{2}(S^1)}=1\}$ and note that
 \[ \Norm{P_q [\tau] - R[\tau]}_{\leb{2}} = \sup_{\Psi \in \rU} \int_{S^1} W'(u) \Psi - \gamma \del_{xx} u \Psi - W'(u)\Psi \d x - a_h^d(\ritz [u],\Psi) =0\]
 by definition of $\ritz [u].$
 As 
 \[\Norm{R [\tau] - \tau}_{\leb{2}(S^1)} = \Norm{P_q [\tau] - \tau}_{\leb{2}(S^1)} \leq C h^{q+1} \Norm{\tau}_{\cont{q+1}(S^1)} \leq C h^{q+1} \Norm{u}_{\cont{q+3}(S^1)} \]
 we find, due to \eqref{dual}, and inverse and trace inequalities, see
 \cite[Lemmas 1.44, 1.46]{dPE12},
 \begin{equation}\label{ee3}
   \begin{split}
     \Norm{P_q[\del_x \tau] - G^+[R[\tau]]}_{\leb{2}} &= \sup_{\Psi \in \rU} \int_{S^1} \del_x \tau \Psi + R[\tau] G^-[\Psi]\d x\\
     &= \sup_{\Psi \in \rU} \sum_{n=0}^{N-1} \qp{ \int_{I_n} (R[\tau] - \tau) \del_x \Psi \d x  + (\tau(x_n) - R[\tau](x_n^-) ) \jump{\Psi}_n}\\
     &\leq \frac{C}{h} \Norm{\tau - R[\tau]}_{\leb{2}(S^1)} \Norm{\Psi}_{\leb{2}(S^1)} \leq C h^q\Norm{u}_{\cont{q+3}(S^1)}.
   \end{split}
 \end{equation}
 Finally we compute, using \eqref{dual}, and inverse and trace inequalities again:
 \begin{equation}\label{ee4}
  \begin{split}
    \Norm{ P_q[\del_{xx} v] - G^+ [G^-[Q[v]]]}_{\leb{2}} 
    &=
    \sup_{\Psi \in \rU} \int_{S^1} \del_{xx} v \Psi + G^-[Q[v]] G^-[\Psi]\d x
    \\
    &=
    \sup_{\Psi \in \rU}\sum_{n=0}^{N-1} \Bigg(  \int_{I_n}\ritz [\del_x v] G^-[\Psi] - \del_x v \del_x \Psi   \d x  + \del_x v(x_n)\jump{\Psi}_n\Bigg)
    \\
    &=
    \sup_{\Psi \in \rU} \sum_{n=0}^{N-1} \Bigg(\int_{I_n} (\ritz [\del_x v] - \del_x v) \del_x \Psi  \d x 
    \\
    &\qquad\qquad\qquad\qquad + \qp{\del_x v(x_n)- \ritz [\del_x v](x_n^-)}\jump{\Psi}_n\Bigg)
    \\
    &=
    \sup_{\Psi \in \rU} \sum_{n=0}^{N-1} \Bigg(\int_{I_n} (\ritz [\del_x v] - S_q^-[\del_x v]) \del_x \Psi  \d x \\
    &\qquad\qquad\qquad + \qp{S_q^-[\del_x v](x_n^-)- \ritz [\del_x v](x_n^-)}\jump{\Psi}_n\Bigg)
    \\
    &\leq
    \sup_{\Psi \in \rU}\frac{C}{h} \Norm{S_q^-[\del_x v] - \ritz [\del_x v]}_{\leb{2}(S^1)}  \Norm{\Psi}_{\leb{2}(S^1)} 
    \\
    &\leq C h^q \Norm{v}_{\cont{q+3}(S^1)}.
  \end{split}
 \end{equation}
 In the last step we used \eqref{eq:proj1}$_1$ and \eqref{Riesz:est}.
 Combining Corollary \ref{Gron} with \eqref{ee1} - \eqref{ee4} we
 obtain the assertion of this Lemma.
\end{proof}

\begin{remark}[Viscosity]
  Note that we need $\mu>0$ only in order to guarantee existence of
  sufficiently regular solutions for small times.  If for $\mu=0$ the
  exact solution is sufficiently regular, all our estimates also hold
  true in this case.
\end{remark}

Using the stability induced by Corollary \ref{cor:Gron3} and the
estimates for the residuals derived in the proof of Theorem \ref{mt}
we have the following estimate with constants independent of $\gamma.$
This result should not be understood as an estimate in the case
$\gamma=0$ but as a uniform estimate in the sharp interface limit case
$\gamma \rightarrow 0.$

\begin{theorem}[Modified entropy error estimate]\label{gamma}
  Let the assumptions of Theorem \ref{mt} be satisfied and let $|W''|$
  be uniformly bounded.  Then, there exists $C>0$ independent of $h,$
  but depending on $q,T,\mu$ such that
  \begin{equation}
    \begin{split}
    &\sup_{0 \leq t \leq T} \bigg(  \Norm{u_h(t,\cdot) - u(t,\cdot)}_{\leb{2}(S^1)}+ \sqrt{\gamma} \norm{u_h(t,\cdot) - u(t,\cdot)}_{\operatorname{dG}} + \Norm{v_h(t,\cdot) - v(t,\cdot)}_{\leb{2}(S^1)}\bigg)\\
    &\qquad + \qp{\mu \int_0^T \norm{v_h(s,\cdot) - v(s,\cdot)}_{\operatorname{dG}}^2 \d s }^{1/2} \\
    & \qquad \qquad \qquad \leq C h^q\bigg( \Norm{u}_{\leb{\infty}(0,T; \cont{q+3}(S^1))} + \Norm{v}_{\leb{\infty}(0,T; \sobh{q+3}(S^1))}+\Norm{\del_t v}_{\leb{\infty}(0,T; \cont{q+2}(S^1))}\bigg) .
    \end{split}
  \end{equation}

\end{theorem}

\begin{remark}[Multiple space dimensions]
The only difficulty in extending the analysis presented here to the multi-dimensional version of the problem
investigated in \cite{Gie} is to construct multi-dimensional discrete gradients with one dimensional kernel.
We need this to be able to find a projection of $v$ which is of optimal order.
It should be noted though, that the aforementioned model is physically inadmissible, and 
probably the multi-dimensional model which should be studied in the future is the Navier-Stokes-Korteweg model.
\end{remark}

\section{Numerical experiments}
\label{sec:num}
In this section we conduct some numerical benchmarking.

\begin{definition}[Estimated order of convergence]
  \label{def:EOC}
  Given two sequences $a(i)$ and $h(i)\downto0$,
  we define estimated order of convergence
  (\EOC) to be the local slope of the $\log a(i)$ vs. $\log h(i)$
  curve, i.e.,
  \begin{equation}
    \EOC(a,h;i):=\frac{ \log(a(i+1)/a(i)) }{ \log(h(i+1)/h(i)) }.
  \end{equation}
\end{definition}

In this test we benchmark the numerical algorithm presented in
\S\ref{sec:ns} against a steady state solution of the
regularised elastodynamics system
(\ref{thorder}) on the domain $\W = [-1,1]$.

We take the double well 
\begin{equation}
  W(u) := \qp{u^2 - 1}^2,
\end{equation}
then a steady state solution to the regularised elastodynamics system
is given by
\begin{gather}
  u(t,x) 
  =
  \tanh\qp{ x \sqrt{\frac{2}{\gamma}}},
  \qquad v(t,x) \equiv 0 \Foreach t.
\end{gather}
For the implementation we are using natural boundary conditions, that is
\begin{equation}
  \partial_x u_h = v_h = 0 \text{ on } [0,T) \times \partial\W,
\end{equation}
rather than periodic.
Tables
\ref{table:p1-gamma-10-3}--\ref{table:p3-gamma-10-3} detail three
experiments aimed at testing the convergence properties for the scheme
using piecewise discontinuous elements of various orders ($p=1$ in
Table \ref{table:p1-gamma-10-3}, $p=2$ in Table
\ref{table:p2-gamma-10-3} and $p=3$ in Table
\ref{table:p3-gamma-10-3}).

\begin{table}[ht]
  \caption{\label{table:p1-gamma-10-3} In this test we benchmark a
    stationary solution of the regularised elastodynamics system using the discretisation
    (\ref{sds}) with piecewise linear elements ($p =
    1$), choosing $k = h^2$. The temporal discretisation is a $2$nd order Crank--Nicolson method. We look at the $\leb{\infty}(0,T; \leb{2}(\W))$ errors
    of the discrete variables $u_h \AND v_h$, the $\leb{\infty}(0,T; dG)$ error of $u_h$ and the $\leb{2}(0,T; dG)$ error of $v_h$. We use $e_u := u - u_h \AND e_v := v - v_h$. In this test we choose
    $\gamma = \mu = 10^{-3}$. We show the rates of convergence for each of the components of the reduced relative and modified entropy error. Notice the leading order terms in the reduced relative entropy error and the modified entropy error converge with the rates in Theorems \ref{mt} and \ref{gamma} respectively.}
  \begin{center}
    \begin{tabular}{c|c|c|c|c|c|c|c|c}
      $N$ & $\Norm{e_u}_{\leb{\infty}(\leb{2})}$ & EOC & $\Norm{e_u}_{\leb{\infty}(dG)}$ & EOC & $\Norm{e_v}_{\leb{\infty}(\leb{2})}$ & EOC & $\Norm{e_v}_{\leb{2}(dG)}$ & EOC \\
      \hline
      16 & 3.033825e-01 & 0.000 & 4.413617e+00 & 0.000 & 2.103556e-01 & 0.000 & 1.928219e+00 & 0.000
      \\
      32 & 2.024675e-01 & 0.583 & 5.051696e+00 & -0.195 & 1.287003e-01 & 0.709 & 1.679159e+00 & 0.200
      \\
      64 & 9.293951e-03 & 4.445 & 3.379746e-01 & 3.902 & 1.392056e-02 & 3.209 & 8.192812e-01 & 1.035
      \\
      128 & 3.226365e-03 & 1.526 & 1.517014e-01 & 1.156 & 4.672567e-03 & 1.575 & 4.290682e-01 & 0.933
      \\
      256 & 1.022094e-03 & 1.658 & 4.636069e-02 & 1.710 & 1.358856e-03 & 1.782 & 2.026073e-01 & 1.083
      \\
      512 & 2.124393e-04 & 2.266 & 9.988999e-03 & 2.215 & 3.129043e-04 & 2.119 & 9.814742e-02 & 1.046
      \\
      1024 & 5.332873e-05 & 1.994 & 2.462207e-03 & 2.020 & 7.765626e-05 & 2.011 & 4.832915e-02 & 1.022
    \end{tabular}
  \end{center}
\end{table}

\begin{table}[ht]
  \caption{\label{table:p2-gamma-10-3} The test is the same as in
    Table \ref{table:p1-gamma-10-3} with the exception that
    we take $p=2$. Notice the leading order terms in the reduced relative entropy error and the modified entropy error converge with the rates in Theorems \ref{mt} and \ref{gamma} respectively.}
  \begin{center}
    \begin{tabular}{c|c|c|c|c|c|c|c|c}
      $N$ & $\Norm{e_u}_{\leb{\infty}(\leb{2})}$ & EOC & $\Norm{e_u}_{\leb{\infty}(dG)}$ & EOC & $\Norm{e_v}_{\leb{\infty}(\leb{2})}$ & EOC & $\Norm{e_v}_{\leb{2}(dG)}$ & EOC \\
      \hline
      16 & 1.582736e-01 & 0.000 & 4.357875e+00 & 0.000 & 8.843701e-02 & 0.000 & 5.622669e-01 & 0.000
      \\
      32 & 1.452531e-02 & 3.446 & 5.367621e-01 & 3.021 & 2.016238e-02 & 2.133 & 1.686844e-01 & 1.737
      \\
      64 & 1.447604e-03 & 3.327 & 1.551374e-01 & 1.791 & 2.482052e-03 & 3.022 & 4.731776e-02 & 1.834
      \\
      128 & 9.269265e-05 & 3.965 & 1.873093e-02 & 3.050 & 4.237385e-04 & 2.550 & 1.427457e-02 & 1.729
      \\
      256 & 7.884262e-06 & 3.555 & 3.723996e-03 & 2.331 & 1.174188e-04 & 1.852 & 3.251624e-03 & 2.001
      \\
      512 & 7.830533e-07 & 3.332 & 8.264272e-04 & 2.172 & 3.262314e-05 & 1.848 & 8.153683e-04 & 1.996
      \\
      1024 & 9.110337e-08 & 3.104 & 2.069450e-04 & 1.998 & 8.486964e-06 & 1.943 & 1.977015e-04 & 2.044
    \end{tabular}
  \end{center}
\end{table}

\begin{table}[ht]
  \caption{\label{table:p3-gamma-10-3} The test is the same as in
    Table \ref{table:p1-gamma-10-3} with the exception that
    we take $p=3$. Notice the leading order terms in the reduced relative entropy error and the modified entropy error converge with the rates in Theorems \ref{mt} and \ref{gamma} respectively.}
  \begin{center}
    \begin{tabular}{c|c|c|c|c|c|c|c|c}
      $N$ & $\Norm{e_u}_{\leb{\infty}(\leb{2})}$ & EOC & $\Norm{e_u}_{\leb{\infty}(dG)}$ & EOC & $\Norm{e_v}_{\leb{\infty}(\leb{2})}$ & EOC & $\Norm{e_v}_{\leb{2}(dG)}$ & EOC \\
      \hline
      16 & 8.127264e-03 & 0.000 & 4.424734e-01 & 0.000 & 1.121668e-02 & 0.000 & 1.469452e-01 & 0.000
      \\
      32 & 4.382422e-03 & 0.891 & 4.035348e-01 & 0.133 & 6.529868e-03 & 0.781 & 1.680196e-01 & -0.193
      \\
      64 & 7.923654e-04 & 2.468 & 7.208921e-02 & 2.485 & 1.112353e-03 & 2.553 & 2.923226e-02 & 2.523
      \\
      128 & 5.081122e-05 & 3.963 & 1.017129e-02 & 2.825 & 1.447472e-04 & 2.942 & 4.334565e-03 & 2.754
      \\
      256 & 2.407321e-06 & 4.400 & 1.270398e-03 & 3.001 & 1.819700e-05 & 2.992 & 5.623497e-04 & 2.946
      \\
      512 & 1.452940e-07 & 4.050 & 1.577331e-04 & 3.010 & 2.338797e-06 & 2.960 & 7.027682e-05 & 3.000
      \\
      1024 & 9.0432415-09 & 4.006 & 1.951425e-05 & 3.015 & 2.936765e-07 & 2.994 & 8.835729e-06 & 2.992
    \end{tabular}
  \end{center}
\end{table}

\bibliographystyle{spmpsci}
\bibliography{nskbib,tristanswritings,tristansbib}


\end{document}